\documentclass{amsart}
\usepackage{amsmath,amsthm,amsfonts,amssymb,latexsym,mathrsfs,graphicx}

\usepackage{hyperref}

\usepackage{enumerate}
\usepackage[shortlabels]{enumitem}
\usepackage{color}
\usepackage {tikz}
\usetikzlibrary {positioning}
\definecolor {processred}{cmyk}{0,0.96,0,0}

\newcommand{\cvd}{\hspace*{\fill}
	{\rm \hbox{\vrule height 0.2 cm width 0.2cm}}}
\renewcommand{\qed}{\cvd}
\newtheorem{theorem}{Theorem}[section]
\newtheorem{corollary}[theorem]{Corollary}
\newtheorem{lemma}[theorem]{Lemma}
\newtheorem{remark}[theorem]{Remark}
\newtheorem{proposition}[theorem]{Proposition}
\theoremstyle{definition}

\tikzstyle{vertex}=[circle, draw, inner sep=0pt, minimum size=6pt]

\newcommand{\Cos}{{\mathrm {Cos}}}

\newcommand{\Aut}{{\mathrm {Aut}}}

\newcommand{\Syl}{\mathrm{Syl}}
\newcommand{\Cay}{\mathrm{Cay}}

\makeatletter
\newcommand*{\rom}[1]{\expandafter\@slowromancap\romannumeral #1@}
\makeatother  
\begin{document}
	
\title{Perfect codes and regular sets in vertex-transitive graphs}
\author[A. Abdollahi]{Alireza Abdollahi$^1$}
\author[Z. Akhlaghi]{Zeinab Akhlaghi$^{2,3}$}
\author[M. Arezoomand]{Majid Arezoomand$^{4}$}

\address{$^{1}$ Department of Pure Mathematics, Faculty of Mathematics and Statistics, University of Isfahan, Isfahan 81746-73441, Iran.}

\address{$^{2}$ Faculty of Mathematics and Computer Science, Amirkabir University of Technology (Tehran Polytechnic), 15914 Tehran, Iran.}

\address{$^{3}$ School of Mathematics,
	Institute for Research in Fundamental Science (IPM)
	P.O. Box:19395-5746, Tehran, Iran. }

\address{$^{4}$ Department of Mathematics, Faculty of Science, Shahid Rajaee Teacher Training
	University, Tehran, 16785-163, I. R. Iran}

%\address{$^{5}$	Department of Mathematics, University of Larestan, Larestan 74317-16137,Iran}

\email{ \newline \text{(A. Abdollahi)} 	a.abdollahi@math.ui.ac.ir  \newline \text{(Z. Akhlaghi) }z\_akhlaghi@aut.ac.ir  \newline \text{(M. Arezoomand) } arezoomand@sru.ac.ir}

%\MSC[MSC 2010]{05C25 · 05C69 · 94B99}

\begin{abstract}
	
A subset \( C \) of the vertex set \( V \) of a graph \( \Gamma = (V,E) \) is termed an $(r,s)$-regular set if each vertex in \( C \) is adjacent to exactly \( r \) other vertices in \( C \), while each vertex not in \( C \) is adjacent to precisely \( s \) vertices in \( C \). A specific case, known as a $(0,1)$-regular set, is referred to as a perfect code. In this paper, we will delve into $(r,s)$-regular sets in the context of vertex-transitive graphs.
It is noteworthy that any vertex-transitive graph can be represented as a coset graph \( \Cos(G,H,U) \). When examining a group \( G \) and a subgroup \( H \) of \( G \), a subgroup \( A \) that encompasses \( H \) is identified as an $(r,s)$-regular set related to the pair \( (G,H) \) if there exists a coset graph \( \Cos(G,H,U) \) such that the set of left cosets of \( H \) in \( A \) forms an $(r,s)$-regular set within this graph. In this paper, we present both a necessary and sufficient condition for determining when a normal subgroup \( A \) that includes \( H \) as a normal subgroup qualifies as an $(r,s)$-regular set for the pair \( (G,H) \). Furthermore, if \( A \) is a normal subgroup of \( G \) containing \( H \), we establish a relationship between \( A \) being a perfect code of \( (G,H) \) and the quotient \( N_A(H)/H \) being a perfect code of \(( N_G(H)/H, {1_{N_{G}(H)/H}}) \).

\end{abstract}
\thanks{The second  author is supported by a Grant from IPM (no. 1404320014).}

%\maketitle

\maketitle

\linespread{1.5}

\section{Introduction}

In this paper, we focus  on finite groups. For a given graph \( \Gamma \), we denote its vertices as \( \mathbf{V}(\Gamma) \) and its edges as \( \mathbf{E}(\Gamma) \). Thus, we  represent \( \Gamma \) as a simple graph \( (\mathbf{V}(\Gamma), \mathbf{E}(\Gamma)) \). If $x$ and $y$ are adjacent for some  $x,y\in {\bf V}(\Gamma)$, we say  $x\sim y$. A subset \( C \) of \( \mathbf{V}(\Gamma) \) is referred to as a perfect code if every vertex in the complement \( \mathbf{V}(\Gamma) \setminus C \) is adjacent to exactly one vertex in \( C \), and there are no edges connecting any two vertices within \( C \). A regular set serves as a generalization of the perfect code concept in graph theory, as outlined in \cite{R1}. Specifically, for non-negative integers \( r\) and \( s \), a subset \( C \) is termed an \( (r, s) \)-regular set within \( \Gamma \) if each vertex in \( \mathbf{V}(\Gamma) \setminus C \) is adjacent to precisely \( s \) vertices in \( C \), while every vertex in \( C \) has exactly \( r \) adjacent vertices also within \( C \). Importantly, a perfect code qualifies as a \( (0, 1) \)-regular set. 
For our graph \( \Gamma = (\mathbf{V}(\Gamma), \mathbf{E}(\Gamma)) \), a partition of \( \mathbf{V}(\Gamma) \) into cells \( \mathcal{V} = \{ V_1, \ldots, V_k \} \) is known as an equitable partition if each cell induces a regular subgraph, and the edges connecting any two distinct cells form a biregular bipartite graph, as detailed in \cite[Section 9.3]{godsil}. In other words, for \( i \neq j \), any vertex \( x \) in \( V_i \) is adjacent to \( b_{ij} \) vertices in \( V_j \), irrespective of the specific choice of \( x \). The quotient matrix corresponding to this partition \( \mathcal{V} \) is defined as a \( k \times k \) matrix \( M = (b_{ij}) \).

Considering a connected \( k \)-regular graph \( \Gamma \), all row sums of the quotient matrix \( M \) equal \( k \). Thus, \( k \) emerges as a simple eigenvalue of \( M \) \cite[Theorem 9.3.3]{godsil}. An equitable partition \( \mathcal{V} \) of \( \Gamma \) is designated as \( \mu \)-equitable if all eigenvalues of its quotient matrix \( M \), excluding \( k \), are consistently equal to \( \mu \). Furthermore, it was established in \cite[Corollary 2.3]{bcg} that a non-trivial coarsening of a \( \mu \)-equitable partition maintains its \( \mu \)-equitable status. Consequently, it becomes essential to thoroughly examine equitable partitions that consist of precisely two parts. In addition, an $(r, s)$-regular set within a $k$-regular graph $\Gamma$ is precisely a completely regular code $C$ in $\Gamma$ (refer to, for instance, \cite{n}). This configuration results in a distance partition that consists of exactly two parts: $\{C, {\bf V}(\Gamma) \setminus C\}$.  An equitable partition with these two parts is also referred to as a perfect 2-coloring \cite{2coloring}. The concept of perfect coloring is a frequent topic of investigation in the realm of coding theory \cite{hamming2coloring1, hamming2coloring2}.
%In this paper, all groups are finite and all graphs are finite, simple and undirected. Let $\Gamma=(V,E)$ be a graph with vertex set $V$ and edge set $E$, and $r,s\geq 0$ be non-negative integers. A subset $C$ of $V$ is called an $(r,s)$-regular set in $\Gamma$ if every vertex in $C$ is adjacent to exactly $r$ vertices of $C$ and every vertex in $V\setminus C$ is adjacent to exactly $s$ vertices of $C$.  A $(0,1)$-regular subset of $\Gamma$ is called a perfect code in $\Gamma$.

A graph with a distance function \( d \) is defined as distance-transitive if for any vertices \( u, v, x, \) and \( y \) where \( d(u, v) = d(x, y) \), there exists an automorphism of the graph that maps \( u \) to \( x \) and \( v \) to \( y \). It is well-known that   Hamming graphs  are distance-transitive.  In the field of coding theory, Lloyd’s Theorem \cite{Lloyd} is an important tool used to show that perfect codes do not exist. Biggs \cite{2} built on Lloyd’s Theorem and applied it to perfect codes found in distance-transitive graphs. This work highlighted that distance-transitive graphs are crucial for studying perfect codes. For more information on perfect codes in these graphs, one can check out the references in \cite{1, 11, 9, 23, 24, 25}.

Distance-transitive graphs are a  subset of vertex-transitive graphs. A graph \( \Gamma \) is termed $G$-vertex-transitive (or vertex-transitive when \( G \) is implicit) if \( G \) is a subgroup of \( \text{Aut}(\Gamma) \) acting transitively on \( {\bf V}(\Gamma) \). Specifically, a $G$-vertex-transitive graph is recognized as a Cayley graph on \( G \) if \( G \) acts regularly on the vertex set, meaning that non-identity elements do not fix any vertex. There has been significant interest in exploring perfect codes and regular sets  within Cayley graphs alongside recent studies in \cite{ HXZ,khodom, paper, subregularset, regularset,note}. 

Although distance-transitive and Cayley graphs are both vertex-transitive, it is noteworthy that certain vertex-transitive graphs do not qualify as distance-transitive or Cayley graphs.  It is a well-established  theorem of Lorimer \cite{Lorimer} stating  a graph is $G$-vertex transitive if and only if it can be expressed as a coset graph \( \text{Cos}(G, H, U) \). 
  Let $G$ be a group and $H\leq G$. Let $U$ be a union of some double cosets of $H$ in $G$ such that $H\cap U=\varnothing$ and $U=U^{-1}$. The coset graph $\Gamma=\Cos(G,H,U)$ on $G/_\ell H$, , the set of left cosets of $H$ in $G$, is a graph defined as follows: the vertex set of $\Gamma$ is $G/_\ell H$ and two vertices $g_1H$ and $g_2H$ are adjacent if and only if $g_1^{-1}g_2\in U$. In the case $H=1$, the coset graph $\Cos(G,H,U)$ is called the Cayley graph of $G$ with respect to $U$ and denoted by $\Cay(G,U)$.  To advance the study of perfect codes in vertex-transitive graphs, in \cite{Wang-Zhang}
 a generalization of the concept surrounding a subgroup  perfect code of a finite group is proposed as follows:  Given a group $G$ and a subgroup $H$ of $G$, a subgroup $A$ of $G$ containing $H$ is called a subgroup perfect code of the pair $(G,H)$ if there exists a coset graph $\Cos(G,H,U)$ such that the set of left cosets of $H$ in $A$ is a  subgroup perfect code in $\Cos(G,H,U)$. 
   In this paper, we  generalize the study of perfect code of coset graphs to regular sets of such a graph and in our study we focus on  regular sets that are subgroups of underlying group of a coset graph.  	Given a group $G$ and a subgroup $H$ of $G$, a subgroup $A$ of $G$ containing $H$ is called an $(r,s)$-regular set of the pair $(G,H)$ if there exists a coset graph $\Cos(G,H,U)$ such that the set of left cosets of $H$ in $A$ is an $(r,s)$-regular set in $\Cos(G,H,U)$. A subgroup $A$ of $G$ is called $(r,s)$-regular set of $G$ if it is an $(r,s)$-regular set of $(G,1)$, equivalently if $A$ is  
   -regular set in a Cayley graph of $G$. 

Let \( G \) be a group and let \( H \leq A \leq G \). By \cite[Theorem 3.1]{Wang-Zhang},  \( A \) is a perfect code of \( (G,H) \) if and only if there exists a left transversal \( X \) of \( A \) within \( G \) such that the remarkable condition \( XH=HX^{-1} \) holds. We wish to extend this magnificent result to  \( (r,s) \)-regular sets through the following  theorem:

%\medskip

\textbf{Theorem A.}  
Let \( H \leq A \leq G \). Then \( A \) is an \( (r,s) \)-regular set of \( (G,H) \) if and only if there exists a subset \( X \subseteq G \) satisfying  \( XH=HX^{-1} \) and \( XH \cap A \) is a union of precisely \( r \) left cosets of \( H \) in \( G\setminus H \), while \( XH \cap tA \) is a union of  precisely   \( s \) left cosets of \( H \) for each \( t \in G\setminus A \).

Let \( A \) be a non-trivial subgroup of \( G \). It is proved in  \cite[Theorem B]{khodom} that for an integer \( r \) and an even integer \( s \), with \( 0 \leq r \leq |A| - 1 \) and \( 0 \leq s \leq |A| \), such that \( \gcd(2, |A| - 1) \) divides \( r \), the subgroup \( A \) is indeed an  \( (r,s) \)-regular set in \( G \). Additionally, by  \cite[Theorem C]{khodom}, \( A \) is a   perfect code of \( G \) if and only if it is an \( (r,s) \)-regular set of \( G \) for a pair of integers \( r,s \) with \( 0 \leq r \leq |A| - 1 \) and \( 0 \leq s \leq |A| \), where \( \gcd(2, |A| - 1) \) divides $r$, and \( s \) is an odd integer. 

It becomes a natural quest to seek such profound relationships within a coset graph of a group, serving as a generalization of the renowned Cayley graph. The subsequent result elegantly shows  this relationship when \( H \unlhd A \unlhd G \):

{\bf Theorem B.} 
Let \( H \unlhd A \unlhd G \) with \( 0 \leq r \leq |A:H| - 1 \) and \( 0 \leq s \leq |A:H| \). The set \( A \) is characterized as an \( (r,s) \)-regular set of \( (G,H) \) if and only if the following conditions are satisfied: 
\begin{itemize}
	\item [{(a)}] \( \gcd(2, |A:H| - 1) \) divides \( r \), 
	\item [{(b)}] \( |H|/|H \cap H^t| \) divides \( s \) for all \( t \in G\setminus A \),
	\item [{(c)}] for every \( x \in G\setminus A \) where \( x^2 \in A \) and \( s|H \cap H^x|/|H| \) is odd, there exists \( a \in A \) such that \( HxaH = H(xa)^{-1}H \).
\end{itemize}
 By Theorem B, in the  case  \( H=1 \), we obtain the following  result: 
 
{\bf Corollary B1.}
	Let \( A \unlhd G \), \( 0 \leq r \leq |A| - 1 \), and \( 0 \leq s \leq |A| \), with \( \gcd(2, |A| - 1) \)  dividing \( r \). Then
	\begin{itemize}
		\item [{(a)}] If \( s \) is even, then \( A \) is an \( (r,s) \)-regular set of \( G \).
		\item [{(b)}] If \( s \) is odd, then \( A \)  is an \( (r,s) \)-regular set of \( G \) if and only if \( xA \) possesses an involution for every \( x \in G\setminus A \) such that \( x^2 \in A \).
	\end{itemize}
%1. If \( s \) is even, then \( A \) is an \( (r,s) \)-regular set of \( G \). \\
%	2. If \( s \) is odd, then \( A \)  is an \( (r,s) \)-regular set of \( G \) if and only if \( xA \) possesses an involution for every \( x \in G\setminus A  \) such that \( x^2 \in A \).

Let \( G \) be a group and \( A \) a normal subgroup of \( G \). On a complementary note,  \cite[Theorem 2.2]{HXZ} states that 
  \( A \)  is a  perfect code of \( G \) if and only if for all \( x \in G \), \( x^2 \in A \) implies    \( (xa)^2 = 1 \) for some \( a \in A \).  
Thus, we derive the following result:

{\bf Corollary B2.}
	Let \( A \unlhd G \), \( 0 \leq r \leq |A| - 1 \) and \( 0 \leq s \leq |A| \), where \( \gcd(2, |A| - 1) \) divides \( r \). If $s$ is odd then 
	\( A \) is an \( (r,s) \)-regular set of \( G \) if and only if \( A \) is a perfect code of \( G \). 
	%the following holds
	%\begin{itemize}
	%	\item [(a)]  If \( s \) is even, then \( A \) is  an \( (r,s) \)-regular set of \( G \). 
	%	\item [(b)] If \( s \) is odd, then \( A \) is an \( (r,s) \)-regular set of \( G \) if and only if \( A \) is a perfect code of \( G \).
	%\end{itemize}
%	1. \\
%	2. 
%\end{{\bf Corollary B2.} 
%}
   
 The result mentioned above is  among the main  results  of \cite{regularset, note}. Thus, Theorem B offers both a generalization and a reproof for those main results of \cite{regularset, note}.  Additionally, we discovered a connection between a normal subgroup \( A \) being a perfect code of \( (G, H) \) and \( N_A(H)/H \) serving as a perfect code of \( N_G(H)/H \). This relationship can be articulated as follows:

\textbf{Theorem C.}
Let $H\leq A\unlhd G$, $0\leq r\leq 
|N_A(H)/H|-1$ with $\gcd(2, |N_A(H)/H|-1)$ divides $r$ and $0\leq s\leq |N_A(H)/H|$. If
$G=N_G(H)A$ and $N_A(H)/H$ is an $(r,s)$-regular set of $N_G(H)/H$, then  $A$ is an $(r,s)$-regular set of $(G,H)$. The converse direction holds if $s=1$.  In particular, $A$ is a perfect code of $(G, H)$ if and only if $G=N_G(H)A$ and $N_A(H)/H$ is a perfect code of $N_G(H)/H$.

\begin{remark}
	The converse of above result is not true in general. To see this, put $G={\rm SL}(2,3)$, $Q_8\cong A\in\Syl_2(G)$ and $C_4\cong H < A$. Then $N_G(H)=A\unlhd G$ and so $G\neq N_G(H)A$.
	Let $H=\langle h\rangle$ and $\langle x\rangle\cong C_3$ be a Sylow $3$-subgroup of $G$. Take $X=\{hx,hx^{-1},x,x^{-1}\}$. Then $XH=xA\cup x^{-1}A=HX^{-1}$ and $XH\cap xA=hx H\cup xH$
	and $xH\cap x^{-1}A=hx^{-1}H\cup x^{-1}H$. This means that $A$ is a $(0,2)$-regular set of $(G,H)$.
\end{remark}

In \cite[Corollary 3.9]{Wang-Zhang} it is sated that:

\begin{lemma}
	 Let $G$ be a group, $H$ a subgroup of $G$ and $A$ a normal subgroup of $G$ such
	that $H \leq A\leq G$. If $A$ is a perfect code of $(G, H)$, then for any $x \in G$ with $x^2\in A$  there
	exists $b \in A$ such that $(xb)^2 \in H$.
\end{lemma} 
As a consequence of Theorem C,  we  prove the following, which shows that the converse of a modification of   the above Lemma   is true:

\textbf{Corollary D. } Let $G$ be a group, $H$ a subgroup of $G$ and $A$ a normal subgroup of $G$ such
that $H \leq A \leq G$. Then  $A$ is a perfect code of $(G, H)$, if and only if $G=AN_G(H)$ and  for any $x \in  G$ with $x^2 \in A$ there
exists $b\in A$ such that $xb\in N_G(H)$ and  $(xb)^2 \in H$.

%\linespread{1.5}
\section{Main results}
\textbf{The proof of Theorem A.}
Let \( A \) be an \( (r,s) \)-regular set of \( (G,H) \). There exists a coset graph \( \Gamma = \Cos(G,H,U) \) such that \( A/_\ell H \) forms an \( (r,s) \)-regular set of \( \Gamma \). According to the definition of a coset graph, we have \( H \cap U = \varnothing \) and \( U^{-1} = U \).
Assuming that \( |G:A| = n \), let \( T = \{1,t_1,\ldots,t_{n-1}\} \) be a left transversal of \( A \) in \( G \). For any \( i \in \{1,\ldots,n-1\} \), it follows that \( t_i \notin A \), which means \( t_i^{-1}H \notin A/_\ell H \). Consequently, \( t_i^{-1}H \) is adjacent to exactly \( s \) elements in \( A/_\ell H \), denoted as \( \alpha_1^iH, \alpha_2^iH, \ldots, \alpha_s^iH \).
%Assume that 
%$\alpha_1^iH,\alpha_2^iH,\ldots,\alpha_s^iH$ are distinct elements in $A/_\ell H$ adjacent to $t_i^{-1}H$. 
There are exactly $r$ elements $\beta_1H, \ldots, \beta_rH$ in $A/_\ell H$ that are adjacent to $H$. We will define two sets: 
$$Y = \{\beta_j \mid 1 \leq j \leq r\},$$ 
and 
$$Z = \bigcup_{i=1}^{n-1} \{t_i \alpha_j^i \mid 1 \leq j \leq s\}.$$
%We claim that $Z$ is a union of $s$ disjoint left transversal of $A$ in $G\setminus A$. %Obviously for each $1\leq i\leq n-1$ we have $\{t_i\alpha_j^i\mid 1\leq j\leq s\}\subseteq t_iA$ and so $|Z\cap t_iA|=s$, {since $t_iA\cap t_jA=\varnothing$ for $i\neq j$}. This proves our claim. Also $|Y\cap A|=r$, {since $Y\subseteq A$}. 
%Thus
It indicates that for all $1 \leq i \leq n-1$, $1 \leq j \leq s$, and $1 \leq k \leq r$, we have $t_i\alpha_j{^i}, \beta_k \in U$. Clearly, $(Z \cup Y)H \subseteq U$. Next, we will demonstrate that $U = (Y \cup Z)H$. 
Assuming, for the sake of contradiction, that there is an element $u \in U \setminus (Y \cup Z)H$. This means we can express $u = ta_0$ for some $t \in T$ and $a_0 \in A$. Therefore, $t^{-1}H$ is adjacent to $a_0H$. If $t = 1$, then it follows that $a_0 = \beta_jh$ for some $1 \leq j \leq r$ and $h \in H$, leading to $u = a_0 \in YH$, which is a contradiction. Thus, we can assume that $t \neq 1$. 
In this case, $t = t_i$ for some $1 \leq i \leq n-1$. Since $t_i^{-1}H$ is adjacent to only $\alpha_i^jH$, it follows that $a_0H = \alpha_i^jH$ for some $1 \leq j \leq s$. Consequently, we have $u = t_i\alpha_i^j h$ for some $1 \leq j \leq s$ and $h \in H$, which implies $u \in ZH$, again leading to a contradiction. 
Thus, we conclude that \( U \subseteq (Y \cup Z)H \), which means \( U = (Y \cup Z)H \). Since \( U = U^{-1} \), we also find that \( (Y \cup Z)H = H(Y \cup Z)^{-1} \). We will now denote \( X := Y \cup Z \), and this completes our proof.

Conversely, assume \( X \) is a subset of \( G \) such that \( XH \cap A = \bigcup_{i=1}^r \alpha_iH \) (a union of \( r \) left cosets of \( H \)), where \( \alpha_iH \neq H \). For each \( t \in G \setminus A \), we have \( XH \cap tA = \bigcup_{i=1}^{s} \beta_i^tH \) (a union of \( s \) left cosets of \( H \)), and \( XH = HX^{-1} \). Therefore, \( XH = HXH = (HXH)^{-1} \). Let \( U = XH \). Then \( U = U^{-1} \). Moreover, \( U \cap H = \varnothing \), since \( \alpha_iH \neq H \) for all \( 1 \leq i \leq r \).

Next, we examine the graph \( \Gamma = \Cos(G,H,U) \). Clearly, every element in \( A/_\ell H \), such as \( a_0H \), is adjacent to exactly \( r \) elements \( a_0\alpha_iH \) for \( 1 \leq i \leq r \). Now let \( t \in G \setminus A \). Then \( t^{-1}A \cap XH = \bigcup_{i=1}^{s} \beta_i^{t^{-1}}H \) implies that \( t\beta_i^{t^{-1}} \in A \) for \( 1 \leq i \leq s \). Thus, \( tH \sim t\beta_i^{t^{-1}}H \) for \( 1 \leq i \leq s \). Consequently, every vertex in \( A/_\ell H \) is adjacent to at least \( r \) vertices in \( A/_\ell H \), and every vertex outside \( A/_\ell H \) is adjacent to at least \( s \) vertices in \( A/_\ell H \). 

Assume \( cH \sim tH \) for some \( c \in A \) and \( t \in G \setminus A \). Then \( t^{-1}c \in XH \cap t^{-1}A \). Hence, there exists \( \beta_j^{t^{-1}}H \subseteq XH \cap t^{-1}A \) such that \( t^{-1}cH = \beta_j^{t^{-1}}H \), where \( 1 \leq j \leq s \).
Assume \( cH \sim tH \) for some \( c \in A \) and \( t \in A \). Then, \( t^{-1}c \in XH \cap A \). Consequently, there exists \( \alpha_iH \subseteq XH \cap A \) such that \( t^{-1}cH = \alpha_iH \), where \( 1 \leq i \leq r \). Therefore, \( tH \) is adjacent to exactly \( s \) elements of \( A/_\ell H \) if \( t \in G \setminus A \), and it is adjacent to exactly \( r \) vertices in \( A/_\ell H \) if \( t \in A \), as desired. This indicates that \( A/_\ell H \) forms an \( (r,s) \)-regular set of \( (G,H) \). \( \blacksquare \)

 \smallskip
 
 If $A$ is an $(r,s)$-regular set of $(G, H)$, then it holds that \(0 \leq r \leq |A:H| - 1\) and \(0 \leq s \leq |A:H|\). We will demonstrate the validity of Theorem B using the following lemmas:
 
 \begin{lemma}\label{key1}
 Let $H\leq A\leq G$. Then  $A$ is an $(r,s)$-regular set of $(G,H)$ if and only if $A$ is a $(0,s)$-regular set and $(r,0)$-regular set of $(G,H)$. 
 \end{lemma}
\begin{proof}
	If \( A \) is an \( (r,s) \)-regular set of \( (G,H) \), then, according to Theorem A, there exists a subset \( X \subseteq G \) such that \( U := XH = HX^{-1} \) forms a union of certain double cosets of \( H \). Furthermore, \( U = U^{-1} \), and for all \( t \in G \setminus A \), the intersection \( U \cap tA \) is a union of \( s \) disjoint left cosets of \( H \). Similarly, \( U \cap A \) is a union of \( r \) disjoint left cosets of \( H \). We can define \( U_0 = U \cap A \) and \( U_1 = U \cap (G \setminus A) \). It is clear that both \( U_0 \) and \( U_1 \) are unions of certain double cosets of \( H \), and we have \( U_0 = U_0^{-1} \) and \( U_1 = U_1^{-1} \). Additionally, \( U_0 \cap tA \) is the empty set for all \( t \in G \setminus A \), while \( U_0 \cap A = U_0 \) represents a union of \( r \) disjoint left cosets of \( H \). In a similar manner, \( U_1 \cap A \) is empty, and \( U_1 \cap tA = U \cap tA \) constitutes a union of \( s \) disjoint left cosets of \( H \). Therefore, \( A \) is recognized as both a \( (r,0) \)-regular set and a \( (0,s) \)-regular set of \( (G,H) \). In fact, \( A \) is a \( (r,0) \)-regular set of \( \Cos(G,H,U_0) \) and a \( (0,s) \)-regular set of \( \Cos(G,H,U_1) \).
	
	Conversely, suppose \( A \) is a \( (r,0) \)-regular set and a \( (0,s) \)-regular set of \( (G,H) \). In this case, there exist inverse-closed subsets \( U_0 \) and \( U_1 \) (which are clearly unions of specific double cosets of \( H \)) such that \( U_0 \cap (A \setminus H) \) is a union of \( r \) disjoint left cosets of \( H \), while \( U_0 \cap (G\setminus A) = \varnothing \). Additionally, \( U_1 \cap tA \) forms a union of \( s \) disjoint left cosets of \( H \), with \( U_1 \cap A = \varnothing \). Consequently, by setting \( U = U_0 \cup U_1 \), we confirm that \( A \) is an \( (r,s) \)-regular set of \( \Cos(G,H,U) \).
	   \end{proof}
\begin{lemma}\label{key2}
	Let $H\unlhd A\unlhd G$ and $0\leq s\leq |A:H|$ be an integer.   Then $A$ is a $(0,s)$-regular set of $(G,H)$ if and only if {both of the following conditions hold}
	\begin{itemize}
		\item [(1)] $|H|/|H\cap H^t|$ divides $ s$ for every $t\in G\setminus A$,
		\item[(2)] for  each  $x\in G\setminus A$ where   $x^2\in A$ and  $s|H\cap H^x|/|H|$ is odd,   there exists $a\in A$ such that $HxaH=H(xa)^{-1}H$.
	\end{itemize}  
\end{lemma}
\begin{proof}
	Let \( t \in G \setminus A \) and \( x \in tA \cup t^{-1}A \). In this case, we find that \( HxH = HatH \) or \( Ha{t^{-1}}H \) for some \( a \in A \). Since \( H \) is normal in \( A \) and \( |H \cap H^t| = |H \cap H^{t^{-1}}| \), it follows that for any \( x \in tA \cup t^{-1}A \), we have:
		\[
	\frac{|HxH|}{|H|} = \frac{|H|}{|H \cap H^t|}.
	\]
	Consequently, this leads us to conclude that \( |HxH| = |HyH| \) for all \( x, y \in tA \cup t^{-1}A \).
	Now, assume \( A \) is a \((0,s)\)-regular set of \((G,H)\) where \( t \in G \setminus A \). By Theorem A, there exists a subset \( X \) such that \( XH = HX^{-1} \). We define \( U := HXH \). Thus, \( U = XH = U^{-1} \) and we find that \( HXH \cap tA = \bigcup_{i=1}^{l}Ht_iH \) for some integer \( l \) and \( t_i \in tA \), since \( A \unlhd G \). Furthermore, \( HXH \cap tA \) consists of exactly \( s \) left cosets of \( H \), leading to the relation:
		\[
	s = l \frac{|Ht_iH|}{|H|} = l \frac{|HtH|}{|H|} = l \frac{|H|}{|H \cap H^t|}.
	\]
		This verifies our desired result.

%	Note that  $HXH\cap A\{t, t^{-1}\}= A\{t, t^{-1}\}A =tA\cup t^{-1}A$ is an inverse-closed set.    $HXH\cap A\{t, t^{-1}\}A=\cup_{i=1}^l H\{t_i, t_i^{-1}\}H$ for some $t_i\in HXH$ and  some integer $l$.
	% So $s=l|H|/|H\cap H^t|$, implying that $|H|/|H\cap H^t|$ divides $s$, for each $t\in G-A$. 
		Let \( x \in G \setminus A \) with \( x^2 \in A \), and suppose that \( s | H \cap H^x| / |H| \) is odd. Therefore, we have \( xA = x^{-1}A = \bigcup_{i=1}^{l} H\{t_i, t_i^{-1}\}H \), where \( t_i \in xA \) and \( l \) is an integer. Given that \( A \) is a \((0,s)\)-regular set, we can assume that \( HXH \cap xA = \bigcup_{i=1}^{l'} H\{t_i, t_i^{-1}\}H \), for some integer \( l' \leq l \). If for all \( 1 \leq i \leq l' \), it holds that \( H t_i H \neq H t_i^{-1} H \), then \( HXH \cap xA \) would represent a union of \( s = 2l' \left( |H|/|H \cap H^x| \right) \) disjoint left cosets of \( H \), which leads to a contradiction under our assumption. Thus, there must exist some \( t_i \) (where \( 1 \leq i \leq l' \)) such that \( H t_i H = H t_i^{-1} H \). It is important to note that \( t_i \in xA \), which implies that there exists \( a \in A \) such that \( t_i = xa \), as required.

	 Assume that \( s \) is divided by \( \frac{|H|}{|H \cap H^t|} \) for each \( t \in G \setminus A \) and if  \( x^2 \in A \) and \( \frac{s|H \cap H^x|}{|H|} \) is odd for some \( x \in G \setminus A \), then there exists \( a \in A \) such that \( HxaH = H(xa)^{-1}H \). We define \( l_t = \frac{s}{\left(\frac{|H|}{|H \cap H^t|}\right)} \).
	 
	 	 Now, assume \( tA \neq t^{-1}A \). We have \( tA = \bigcup_{i=1}^l Ht_iH \) and \( t^{-1}A = \bigcup_{i=1}^l Ht_i^{-1}H \) for some \( t_i \in tA \) and an integer \( l \). It is clear that \( l_t \leq \frac{|A:H|}{\left(\frac{|HtH|}{|H|}\right)} = \frac{|A|}{|HtH|} = l \). 
	 	 Next, we define \( U_t = \bigcup_{i=1}^{l_t} H \{t_i, t_i^{-1}\}H \). Notice that \( U_t^{-1} = U_t \) and \( U_t \cap tA \) is a union of \( l_t \times \frac{|H|}{|H \cap H^t|} = s \) disjoint left cosets of \( H \).
	
	Now,  assume that \( tA = t^{-1}A = \bigcup_{i=1}^{l} H\{t_i, t_i^{-1}\}H \), where \( t_i \in tA \) and \( l \) is an integer. It is possible that for some \( i \), we have \( Ht_iH = Ht_i^{-1}H \). Without loss of generality, we can assume that there exists \( m \leq l \) such that for \( 1 \leq i \leq m \), \( Ht_iH \neq Ht_i^{-1}H \), and for \( i > m \), \( Ht_iH = Ht_i^{-1}H \). 
	
	First, suppose that \( l_t \) is an even integer. If \( l_t < 2m \), we define \( U_t = \bigcup_{i=1}^{l_t/2} H\{t_i, t_i^{-1}\}H \) and if \( l_t > 2m \), we set \( U_t = \bigcup_{i=1}^{m} H\{t_i, t_i^{-1}\}H \cup \bigcup_{i=m+1}^{l_t-2m} Ht_iH \). In both scenarios, we find that \( U_t = U_t^{-1} \) and that \( U_t \) is a union of \( s = \frac{l_t \cdot |H|}{|H \cap H^t|} \) disjoint left cosets of \( H \).	
	Next, consider the case when \( l_t \) is an odd integer. According to our assumption, there exists \( a \in A \) such that \( HtaH = H(ta)^{-1}H \). This indicates that \( m < l \), allowing us to assume that \( Ht_lH = Ht_l^{-1}H \). If \( l_t \leq 2m + 1 \), we define \( U_t = \left(\bigcup_{i=1}^{(l_t-1)/2} H\{t_i, t_i^{-1}\}H\right) \cup Ht_lH \). On the other hand, if \( l_t > 2m + 1 \), we set \( U_t = \left(\bigcup_{i=1}^{m} H\{t_i, t_i^{-1}\}H\right) \cup \bigcup_{i=m+1}^{l_t-2m} Ht_iH \). In both cases, we again find \( U_t = U_t^{-1} \), and \( U_t \) is a union of \( s \) disjoint left cosets of \( H \).
	
	We know that \( G = \bigcup_{t \in T} (tA \cup t^{-1}A) \) for some subset \( T \subseteq G \), and that \( (tA \cup t^{-1}A) \cap (yA \cup y^{-1}A) = \varnothing \) for each \( y \neq t \in T \). We then define \( U = \bigcup_{t \in T \setminus A} U_t \). It is clear that \( U^{-1} = U \), \( U \cap A = \varnothing \), and consequently \( U \cap H = \varnothing \). Furthermore, since \( U \cap (tA \cup t^{-1}A) = U_t \), we conclude that \( U \cap tA \) is a union of \( s \) disjoint left cosets of \( H \). Therefore, we establish that \( A \) is a \( (0,s) \)-regular set of \( (G,H) \), thus completing the proof.
\end{proof}
\begin{lemma}\label{key3}
	Let $H\unlhd A \leq G$ and $0\leq r\leq |A:H|-1$  is an integer. Then $A$ is {an} $(r,0)$-regular set of $(G,H)$ if and only if $\gcd(2, |A: H|-1)$ divides $r$. 
	\end{lemma}
\begin{proof}Let $\overline{A} = A/H$ and let $e = \overline{H}$ be the trivial element of $\overline{A}$. First, we assume that $A$ is an $(r, 0)$-regular set of $(G, H)$. This implies the existence of a set $U$, which is an inverse-closed union of some double cosets of $H$. Furthermore, since $U \cap (G \setminus A) = \varnothing$, it follows that $U \subseteq A$. Consequently, $U$ must be a union of $r$ disjoint left cosets of $H$. Thus, we have $U \subseteq \overline{A}$, where $U \cap H = \varnothing$ and $U$ is inverse-closed. This means that $\overline{A} \setminus \{e\}$ contains an inverse-closed set with $r$ elements. If $|\overline{A}|$ is odd, then all non-trivial elements of $\overline{A}$ have odd order, and since $U$ is an inverse-closed subset of $\overline{A} \setminus \{e\}$, we deduce that $U$ is a union of an even number of elements of $\overline{A}$. Therefore, $2 = \gcd(2, |\overline{A}| - 1)$ divides $r$. If $|\overline{A}|$ is even, it is evident that $r$ is  divided by $1 = \gcd(2, |\overline{A}| - 1)$.
	
	Now, we need to prove the converse direction. Assume $|\overline{A}|$ is even, which implies that $\overline{A}$ has an involution. Let $I$ denote the set of all involutions of $\overline{A}$. To complete the proof, it suffices to demonstrate that for every integer $0 \leq r \leq |\overline{A}| - 1$, there exists an inverse-closed subset $U$ in $\overline{A} \setminus \{e\}$ of size $r$. In fact, if $r$ is an even integer less than $|\overline{A} \setminus (\{e\} \cup I)|$, then $U$ can be any inverse-closed subset of $\overline{A} \setminus (\{e\} \cup I)$. If $r > |\overline{A} \setminus (\{e\} \cup I)|$ is an even integer, we can take $U = \overline{A} \setminus (\{e\} \cup I) \cup J$, where $J \subset I$ has size $r - |\overline{A} \setminus (\{e\} \cup I)|$. If $r \leq |\overline{A} \setminus (\{e\} \cup I)| + 1$ is odd, we can choose an inverse-closed subset $U_0 \subseteq \overline{A} \setminus (\{e\} \cup I)$ of size $r - 1$ and set $U = U_0 \cup \{\alpha\}$, where $\alpha \in I$ is an arbitrary element. If $r \geq |\overline{A} \setminus (\{e\} \cup I)| + 1$ is odd, we set $U = (\overline{A} \setminus (\{e\} \cup I)) \cup J$, where $J \subseteq I$ is any arbitrary subset of size $r - |\overline{A} \setminus (\{e\} \cup I)|$. In the case when $|\overline{A}|$ is odd, a similar argument leads to the conclusion that all inverse-closed subsets of $\overline{A} \setminus \{e\}$ have even size. Furthermore,  in this case  for all even integers $r \leq |\overline{A} \setminus \{e\}| = |A : H| - 1$, we can find at least one inverse-closed subset $U$ of size $r$.
\end{proof}

{\bf The proof of Theorem B. }
	It is straightforward using Lemmas \ref{key1},  \ref{key2} and \ref{key3}. $\blacksquare$ 

\smallskip 

We employ the following lemma to demonstrate Theorem C. It is important to note that if we include the assumption of normality for \( H \) in the statement of the lemma below, it can be derived as a consequence of Theorem B.  

\begin{lemma}\label{khodom}(see \cite[Theorem B]{khodom})
	Let $H$ be a nontrivial subgroup of $G$. Then, for an integer $r$ and an even
	integer $s$, with $0 \leq  r \leq |H|-1$ and $0 \leq  s \leq  |H|$, such that $\gcd(2, |H|-1)$ divides
	$r$, $H$ is an $(r, s)$-regular set in $G$. 
\end{lemma}

\smallskip  

\textbf{The proof of Theorem C.}
Assume first that \( G = N_G(H)A \) and that \( N_A(H)/H \) is an \( (r,s) \)-regular set of \( N_G(H)/H \). Let \( T \subseteq N_G(H) \) be a left transversal of \( A \) in \( G \). Consequently, \( T \) serves as a left transversal of \( N_A(H) \) in \( N_G(H) \), which implies that \( \{tH \mid t \in T\} \) is a left transversal of \( N_A(H)/H \) in \( N_G(H)/H \). 
Since \( N_A(H)/H \) is an \( (r,s) \)-regular set of \( N_G(H)/H \), according to Theorem A, there exists \( XH \subset N_G(H)/H \) such that \( (XH)^{-1} = XH \) and \( XH - N_A(H)/H \) is a union of \( s \) disjoint left transversals of \( N_A(H)/H \) in \( N_G(H)/H - N_A(H)/H \). Additionally, the intersection \( XH \cap N_A(H)/H \) consists of \( r \) left cosets of \( H \). 
In other words, for each \( t \in T \setminus A \), we have \( |XH \cap (tH)N_A(H)/H| = s \) and \( |XH \cap N_A(H)/H| = r \). It is important to note that \( N_A(H)/H \subseteq A/_\ell H \) and \( tN_A(H)/H \subseteq t(A/_\ell H) \). Thus, since \( XH \subset N_G(H)/H \), we can conclude that \( XH \cap tA \) contains at least \( s \) left cosets of \( H \), denoted as \( \{x_1H, \dots, x_sH\} \), where \( t \in G \setminus A \). 
If \( XH \cap tA \) contained more than \( s \) left cosets of \( H \), there would exist \( x_{s+1} \in X \) such that \( x_{s+1}H \subseteq tA \). Furthermore, \( x_{s+1}H \) must belong to \( t_0N_A(H)/H \) for some other \( t \neq t_0 \in T \), since \( XH \) is a subset of \( N_G(H)/H \) and \( T \) is a left transversal of \( N_A(H) \) in \( N_G(H) \). This leads to \( x_{s+1}H \subseteq (tA \cap t_0A) \), which is a contradiction, as \( tA \cap t_0A = \varnothing \). 
Therefore, for each \( t \in G \setminus A \), \( XH \cap tA \) is precisely a union of \( s \) left cosets of \( H \). Similarly, \( XH \cap A \), by the same reasoning, is a union of exactly \( r \) left cosets of \( H \). This indicates that \( A \) is an \( (r,s) \)-regular set of \( (G,H) \), thereby completing the proof of the first part.

Assume that \( A \) is an \( (r,1) \)-regular set of \( (G,H) \).  According to Theorem A, there exists a subset \( X \subseteq G \) such that \( XH = HX^{-1} \), and \( XH \cap A \) is a union of \( r \) distinct left cosets of \( H \). Additionally, for each \( t \in G \setminus A \), \( XH \cap tA \) is a left coset of \( H \). Let us define \( XH \cap A = \bigcup_{i=1}^r \alpha_i H \) and \( XH \cap tA = \beta^t H \), where \( \alpha_i, \beta^t \in X \) for \( 1 \leq i \leq r \) and \( t \in G \setminus A \). 

Next, let \( T = \{1, t_1, t_2, \ldots, t_n\} \) be a left transversal of \( A \) in \( G \). Then, we have 
\[ 
X = \{\alpha_i\}_{i=1}^r \cup \{\beta^{t_j}\}_{j=1}^n. 
\] 
We need to show that \( \beta^{t_j} \in N_G(H) \) for \( 1 \leq j \leq n \). Without loss of generality, assume that \( \beta^{t_1} \notin N_G(H) \). This leads to \( H(\beta^{t_1})^{-1}H \neq (\beta^{t_1})^{-1}H \), with \( H\beta^{t_1}H \subseteq t_1A \) and \( H(\beta^{t_1})^{-1}H \subseteq At_1^{-1} = t_1^{-1}A \). 
For each \( h \in H \), we find that 
\[ 
|(\beta^{t_1}H)^{-1} \cap (\beta^{t_1})^{-1}H| = |(\beta^{t_1}H)^{-1} \cap h(\beta^{t_1})^{-1}H| = |H \cap H^{\beta^{t_1}}| \geq 1. 
\]
Let \( h(\beta^{t_1})^{-1}b \in (\beta^{t_1}H)^{-1} \cap h(\beta^{t_1})^{-1}H \) for some \( b \in H \). Then, we have 
\[ 
h(\beta^{t_1})^{-1}bH \subseteq H(\beta^{t_1})^{-1}H \subseteq HXH \cap t_1^{-1}A = XH \cap t_1^{-1}A. 
\] 
Thus, for every \( h \in H \), it follows that 
\[ 
h(\beta^{t_1})^{-1}bH = h(\beta^{t_1})^{-1}H = (\beta^{t_1})^{-1}H. 
\] 
Consequently, we arrive at the contradiction 
\(
H(\beta^{t_1})^{-1}H = (\beta^{t_1})^{-1}H. 
\)
Therefore, we conclude that \( \beta^{t_1} \in N_G(H) \), as required.
 %%%%%%%%%%%%%%%%%%%%%%%%%%%%%%%%%%%%%%%%%%%%%%%
 % As $\beta_i^{t_1}=ta_i$ for some $a_i\in A$.
 %  Then $(\beta_i^{t_1})^{-1}=a_i^{-1}t^{-1}\in t_1^{-1}A$. Hence $\bigcup_{i=1}^s(\beta_i^{t_1})^{-1}H\cup h(\beta_1^{t_1})^{-1}bH\subseteq XH\cap t_1^{-1}A$. It is not possible, unless some of those left cosets are equal. First assume $a_it^{-1}H=(\beta_i^{t_1})^{-1}H=(\beta_j^{t_1})^{-1}H=a_jt^{-1}H$. \textcolor{red}{This implies $a_j^{-1}a_i\in t^{-1}H\subseteq t^{-1}A$, a contradiction.} 
% Now let $(\beta_i^{t_1})^{-1}H=h(\beta_1^{t_1})^{-1}bH=h(\beta_1^{t_1})^{-1}H$. Then $a_i^{-1}t^{-1}H=ha_1^{-1}t^{-1}H$ which means $a_1h^{-1}a_i^{-1}\in t^{-1}H\subseteq t^{-1}A$, which is a contradiction. Hence $H\beta_1^{t_1}H=\beta_1^{t}H$ and it does not contain any other left cosets $h\beta_1^{t_1}H$ which means that $\beta_1^{t_1}\in N_G(H)$. 
Thus, since \(X\) contains a left transversal of \(A\), we can conclude that \(G=N_G(H)A\). We may assume \(T \subseteq N_G(H)\). Therefore, the set \(\{tH \mid t \in T\}\) serves as a left transversal of \(N_A(H)/H\) within \(N_G(H)/H\). To complete the proof, we need to demonstrate that \(|XH \cap (tN_A(H))/H|=1\) and \(|XH \cap N_A(H)/H|=r\) for each \(t \in T \setminus A\). It is evident that \(tN_A(H)/H \subseteq tA/_{\ell} H\) for every \(t \in T\). Hence, \(XH \cap tN_A(H)/H\) contains at most one left coset of \(H\) when \(t \in T \setminus A\). 
On the other hand, for each \(t \in T \setminus A\), we have \(XH \cap tA = \beta^tH\), where \(\beta^t \in N_G(H)\). Thus, \(\beta^t = ta \in N_G(H)\) for some \(a \in A\). Given that \(t \in N_G(H)\), it follows that \(a \in N_A(H)\). Consequently, \(\beta^tH = taH \subseteq tN_A(H)\), which implies that \(|XH \cap tN_A(H)/H|= 1\) for every \(t \in T \setminus A\).
Furthermore, by Lemma \ref{khodom}, we know that \(N_A(H)/H\) is a \((r,0)\)-regular subset of \(N_G(H)/H\). Therefore, there exists an inverse closed set \(YH\) of \(N_G(H)/H\) such that \(YH\) is a union of \(r\) disjoint left cosets of \(H\) in \(N_A(H)/H\). We define \(X_0H = YH \cup \{\beta^{t_j}H\}_{j=1}^n\). 
It is noteworthy that \((XH \cap (G \setminus A))^{-1} = (\{\beta^{t_j}H\}_{j=1}^n)^{-1} = HX^{-1} \cap (G \setminus A) = XH \cap (G \setminus A) = \{\beta^{t_j}H\}_{j=1}^n\). Thus, we have \((X_0H)^{-1} = (YH)^{-1} \cup (\{\beta^{t_j}H\}_{j=1}^n)^{-1} = HY^{-1} \cup (\{\beta^{t_j}H\}_{j=1}^n)^{-1}= YH \cup \{\beta^{t_j}H\}_{j=1}^n = X_0H\).
In conclusion, we determine that \(N_A(H)/H\) is an \((r,1)\)-regular set of \(N_G(H)/H\), as required. $\blacksquare$
 
 \begin{corollary}
 	Let $H\unlhd A\unlhd G$. Then $A$ is a perfect code of $(G,H)$ if and only if  $H\unlhd G$ and $A/H$ is a perfect code of $G/H$ 
 	\end{corollary}
 \begin{proof}
 	According to Theorem C, we have  $A$ is a perfect code of $(G,H)$ if and only if $G=N_G(H)A=N_G(H)$ and $A/H=N_A(H)/H$  is a perfect code of $G/H$. 
 	\end{proof}

\begin{corollary}
	Let $H\leq A\unlhd G$ and $|A|$ or $|G:A|$ is odd. Then $A$ is a perfect code of $(G,H)$ if and only if $G=N_G(H)A$.
\end{corollary}
\begin{proof}
In this case $|N_A(G)/H|$ or $|N_G(H): N_A(H)|$ is odd and so the result follows from Theorem C and \cite[Corollary 2.3]{HXZ}.
\end{proof}

We use the following lemma in the proof of Corollary D:

\begin{lemma}\label{asli}
	(see \cite[Theorem 2.2]{HXZ})
	Let $G$ be a group and $H$ a normal subgroup of $G$. Then $H$ is a perfect
	code of $G$ if and only if for all $x \in G$, $x^2 \in H$ implies $(xh)^2 =1$  for some $h \in  H$.
\end{lemma}

Corollary D establishes the validity of the converse of a modification found in \cite[Corollary 3.9]{Wang-Zhang}.

\smallskip

\textbf{The proof of Corollary D.}\label{nag}
%	Let $H\leq A\unlhd G$. Then $A$ is a perfect code of $(G,H)$ if and only if $N_G(H)A=G$ and for all $x\in G$ with $x^2\in A$ there exists $b\in A$ such that $xb\in N_G(H)$  and $(xb)^2\in H$.
Let \( A \) be a perfect code of \( (G,H) \).   From Theorem C, it follows that \( G = AN_G(H) \). Furthermore, according to \cite[Corollary 3.9]{Wang-Zhang}, if \( x \in G \) and \( x^2 \in A \), then there exists \( b \in G \) such that \( (xb)^2 \in H \). To complete the proof, we only need to show that \( xb \in N_G(H) \). Remind that, there exists a left transversal of $A$, say $X\subseteq N_G(H)$ such that $XH=HX^{-1}$ is a union of some double cosets of $H$ and $XH\cap tA$ is a coset of $H$. Moreover,  Since \( A \) is a normal subgroup of \( G \), the set \( A\{x,x^{-1}\}A = xA \) must contain an element \( xb \) such that \( xbH = HxbH = H(xb)^{-1}H\subseteq XH \). This means that \( xb \in N_G(H) \), and thus we have completed the proof.   
	%	 Then $Hx^{-1}H\neq x^{-1}H$. On the other hand, $HxH\subseteq xA$ and $Hx^{-1}H\subseteq x^{-1}A$. Let $H=\{h_1,h_2,\ldots,h_n\}$, where $h_1=1$. Then $Hx^{-1}H=\bigcup_{i=1}^n h_ix^{-1}H$. Let $h=h_2$. 
	%\textcolor{red}{Then
	%	\[ |(xH)^{-1}\cap x^{-1}H|=|(xH)^{-1}\cap hx^{-1}H|=|H\cap H^x|.\]}
	%	If $hx^{-1}b\in (xH)^{-1}\cap hx^{-1}H$, then $(xH)^{-1}$ contains both $hx^{-1}b$ and $x^{-1}$. This implies that $XH\cap x^{-1}A$ contains both $x^{-1}$ and $hx^{-1}b$, which is a contradiction. Hence \textcolor{red}{$X\subseteq N_G(H)$}, which means $G=N_G(H)A$. 
		%Now we are going to prove the converse of the corollary. Let $x\in G$ be an arbitrary element. If $x\in A$, we are done. Hence we may assume that $x^2\in A$ and $x\notin A$. Thus $xA=x^{-1}A$. We know that $xA\cap XH=tH$ for some $t\in X$ and so $x^{-1}A\cap XH=tH$. Furthermore, $(tH)^{-1}\subseteq XH=HX^{-1}$. On the other hand, $(tH)^{-1}=Ht^{-1}\subseteq Ax^{-1}=xA$. If
	%there exists $l\in Ht^{-1}\setminus tH$, then $lH\cup tH\subseteq XH\cap xA$, which is a contradiction. Thus $Ht^{-1}=tH$ and so $t^{-1}=th$ for some $h\in H$, implying that $t^2\in H$. Since $t=xb$ for some $b\in A$ and $t\in X\subseteq N_G(H)$, the proof of one direction is complete.
		Conversely, suppose that \( G = N_G(H)A \) and for every \( x \in G \) with \( x^2 \in A \), there exists \( b \in A \) such that \( xb \in N_G(H) \) and \( (xb)^2 \in H \). In this case, we can conclude that \( N_A(H)/H \subseteq N_G(H)/H \). Furthermore, for every \( xH \in N_G(H)/H \) where \( (xH)^2 \in N_A(H)/H \), there exists \( b \in A \) such that \( xb \in N_G(H) \) and \( (xbH)^2 = H \). Since both \( x \) and \( xb \) are in \( N_G(H) \), it follows that \( b \in N_A(G) \), and consequently \( bH \in N_A(H)/H \). According to Theorem \ref{asli}, we find that \( N_A(H)/H \) is a perfect code of \( N_G(H)/H \), which allows us to conclude the proof, using Theorem C.  \qed
		% there exists $T\subseteq N_G(H)$  where $T$ is a left transversal of $A$ in $G$ such that   $TH=HT^{-1}$. Then  we  $TH=HT=HT^{-1}$. If there exists $t\in T\setminus T^{-1}$, then we have two possibilities
%	$t^{-1}A=tA$ or $t^{-1}A\neq tA$. \textcolor{red}{In the later case, since $t\in N_G(H)$, we have $t^{-1}\in N_G(H)$ and so we can replace $\{b\}=T\cap t^{-1}A$ by $t^{-1}$ and we are done.} In the first case, we have $t^2\in A$. Thus there exists $b\in A$ such that $tb\in N_G(H)$ and $(tb)^2\in H$. Hence we replace $t$ by $tb$, and so we have $tbH=Htb$ and $TH=HT=HT^{-1}$. This completes the proof. $\blacksquare$
%\end{proof}
%%%%%%%%%%%%%%%%
\begin{corollary}
	Let $A$ be a normal subgroup of $G$ and $H\in\Syl_p(A)$. Then $A$ is a perfect code of $(G,H)$ if and only if for every $x\in G\setminus A$ with $x^2\in A$, there exists $a\in A$ such that $xa\in N_G(H)$ and $(xa)^2\in H$.
\end{corollary}
\begin{proof}
	As $H$ is a Sylow $p$-subgroup of $G$, then by well-known Frattini  argument we have that $G=N_G(H)A$ and so by Corollary D we get the result. 
\end{proof}

The following lemma and corollary give us  necessary conditions for a subgroup to be a perfect code, in case when we do not have the assumption of normality of $A$.  

\begin{lemma}\label{lem1}
	Let $H\leq A\leq G$ and $A$ be a perfect code of $(G,H)$. Then for every $g\in G$ there exists $a\in A$ such that $A^{ga}\cap H=H^{ga}\cap H$.
\end{lemma}
\begin{proof}
	Let $g\in G$.	As $A$ is a perfect code, by Theorem A, there exists a left transversal $X$ of $A$ in $G$ such that $XH=HX^{-1}$. Thus there exist $x\in X$ and $a\in A$ such that $x=ga$. Clearly, $H^{x}\cap H\subseteq A^{x}\cap H$. Suppose, towards a contradiction, that $A^{x}\cap H\neq H^{x}\cap H$. Then there exists $b\in A^{x}\cap H\setminus (H^{x}\cap H)$. Thus,  $x^{-1}, bx^{-1}\in x^{-1}A\cap Hx^{-1}$. Since $XH=HX^{-1}$, we conclude that $x^{-1}H\cup bx^{-1}H\subseteq (XH\cap x^{-1}A)$. On the other hand, $X$ is a left transversal of $A$ in $G$ and so we have $|XH\cap x^{-1}A|=|H|$, which implies that $x^{-1}H=bx^{-1}H$. Hence $b\in H^x\cap H$, a contradiction. This completes the proof.
\end{proof}

\begin{corollary}
	Let $H\unlhd A\leq G$ and $A$ is a perfect code of $(G,H)$. Then for every $x\in G$ we have $|HA^x|\mid |AA^x|$. 
\end{corollary}
\begin{proof}
If \( x \in A \), we have nothing further to prove. Thus, let \( x \in G \setminus H \). Then \( A \{x, x^{-1}\} A \) represents a union of some double cosets of \( H \), which we can express as \( A \{x, x^{-1}\} A = \bigcup_{i=1}^{l} H \{t_i, t_i^{-1}\} H \), where \( l \) is an integer and \( t_i \in A \{x, x^{-1}\} A \). Since \( H \unlhd A \), we find that \( H t_i H = H a_i x b_i H \) or \( H a_i x^{-1} b_i H \) for some \( a_i, b_i \in A \).
Furthermore, we have either 
\[
\frac{|H t_i H|}{|H|} = \frac{|H|}{|H \cap H^{t_i}|} = \frac{|H|}{|H \cap H^{a_i x b_i}|} = \frac{|H|}{|H^{b_i^{-1}} \cap H^x|} = \frac{|H|}{|H \cap H^x|}
\]
or, similarly, 
\[
\frac{|H t_i H|}{|H|} = \frac{|H|}{|H \cap H^{x^{-1}}|}.
\]
On another note, it follows that \( |H \cap H^x| = |H \cap H^{-x}| \). 
Thus, we conclude that \( \frac{|H t_i H|}{|H|} = \frac{|H t_j H|}{|H|} = \frac{|H|}{|H \cap H^x|} \) for all \( 1 \leq i, j \leq l \). In the sequel, we aim to demonstrate that \( \frac{|H|}{|H \cap H^x|} \) divides \( \frac{|A|}{|A \cap A^x|} \) for each \( x \in G \setminus A \).
	As \( A \) is a perfect code of \( (G,H) \), there exists a left transversal of \( A \) in \( G \), denoted as \( X \), such that \( XH = HX^{-1} \). Therefore, we can conclude that \( XH = HXH = HX^{-1}H = HX^{-1} \). This leads us to \( XH \cap A\{x, x^{-1}\}A = HXH \cap A\{x,x^{-1}\}A = HZH \), where \( Z \subseteq A\{x,x^{-1}\}A \).	
	In fact, we have \( HXH \cap A\{x, x^{-1}\}A = \bigcup_{i=1}^k Hx_iH \), for some natural number \( k \) and elements \( x_i \in A\{x,x^{-1}\}A \). Since \( HXH \cap A\{x, x^{-1}\}A \) is inverse-closed, we can assume that \( \bigcup_{i=1}^k Hx_iH = \bigcup_{i=1}^{s} H\{t_i, t_i^{-1}\}H \), for some \( s \leq l \). Thus, we can assume \( Z = \{\{t_i, t_i^{-1}\}\}_{i=1}^s \).
		Finally, we find that \( \frac{|HZH|}{|H|} = \frac{|Z||H|}{|H \cap H^x|} \). %\textcolor{red}{So $HZH=\bigcup_{i=1}^{|Z|}r_iH=XH\cap A\{x,x^{-1}\}A$}.
	On the other hand, we have $XH\cap A\{x,x^{-1}\}A=LH$, where  
	\[
	|L|=\frac{|A\{x,x^{-1}\}A|}{|A|}=\frac{2|AxA|}{|A|}=\frac{2|A|}{|A\cap A^x|},
	\]
	if $AxA\neq Ax^{-1}A$, and 
	\[
	|L|=\frac{|A\{x,x^{-1}\}A|}{|A|}=\frac{|AxA|}{|A|}=\frac{|A|}{|A\cap A^x|},
	\]
	if $AxA=Ax^{-1}A$. 
		First,  assume $AxA\neq Ax^{-1}A$. In this case, since $Z=Z^{-1}$, we observe that $|Z|=2k$, because the inverse of each element in $AxA$ belongs to $Ax^{-1}A$. Therefore, we find that 
	$$|L|=\frac{2|A|}{|A\cap A^x|}=|Z|\frac{|H|}{|H\cap H^x|}=2k \frac{|H|}{|H\cap H^x|},$$ 
	which indicates that $\frac{|H|}{|H\cap H^x|}$ divides $\frac{|A|}{|A\cap A^x|}$. 
		Now, consider the case where $AxA=Ax^{-1}A$. In this situation, we have 
	$$|L|=\frac{|A|}{|A\cap A^x|}=|Z|\frac{|H|}{|H\cap H^x|},$$ 
	and once again, $\frac{|H|}{|H\cap H^x|}$ divides $\frac{|A|}{|A\cap A^x|}$. 
		Moreover, according to Lemma \ref{lem1}, there exists an element $a\in A$ such that $H\cap H^{xa}=H\cap A^{xa}$. This suggests that $|H\cap H^x|=|H\cap A^x|$, meaning that $\frac{|H|}{|H\cap A^x|}$ divides $\frac{|A|}{|A\cap A^x|}$, and thus $|HA^x|\mid |AA^x|$.
	%
	%Now suppose that $AxA=Ax^{-1}A$. Then $A\{x,x^{-1}\}A=AxA$. The proof is the same as the previous case just in this case $|Z|$ might be odd or even and \textcolor{red}{$\frac{|AxA|}{|A|}=\frac{|A|}{AxA}$} and $\frac{|H|}{|H\cap A^x|}$ divides $\frac{|A|}{|AxA|}$.
\end{proof}

 By the following result, one can completely determine the structure of arc-transitive graphs.  
\begin{lemma}(\cite[Theorem 3.2.8]{Symmetry})
	For a vertex-transitive digraph $\Gamma$ with $G\leq\Aut(\Gamma)$ a transitive subgroup, $\Gamma$ is $G$-arc-transitive if and only if $\Gamma\cong\Cos(G,H,S)$, where $H=G_u$ is the stabilizer of $u$ in $G$, for some $u\in {\bf V}(\Gamma)$, and $S=HsH$ is a single double coset of $H$ for some $s\in G$.
\end{lemma}
 Building upon the aforementioned proposition, we arrive at a definitive set of necessary and sufficient conditions for a subgroup to qualify as a perfect code within the context of an arc-transitive graph. Notably, these conditions encompass the hypothesis addressed in Lemma \ref{lem1}.     
\begin{proposition}
	Let $\Gamma=\Cos(G,H,S)$, where $H=G_u$ for some $u\in V(\Gamma)$ and $S=HxH$ for some $x\in G$ is a $G$-arc-transitive graph and $H\leq A\leq G$. Then $S=S^{-1}$, and the set of left cosets of $H$ in $A$ is a perfect code of $\Gamma$ if and only if the following hold
	\begin{itemize}
		\item [(1)] $G=A\cup AxA$,
		\item[(2)] $H\cap  H^{x}= H\cap  A^{x}$,
		\item[(3)] for each $ a\in A$ there exists $h\in H$ such that $ha\in A\cap A^{x}$.
	\end{itemize}
	\end{proposition}
	\begin{proof}
			First, suppose that the set of left cosets of \( H \) in \( A \) is a perfect code of \( \Gamma \). Let \( y \in G \setminus A \) such that \( y \in G \setminus A \). Then \( S \cap yA \) is equal to a left coset of \( H \), which means there exists \( a \in A \) such that \( yaH = S \cap yA \subseteq S \subseteq AxA \). Therefore, \( y \in AxA \). This implies that \( G = A \cup AxA =A\cup Ax^{-1}A\). 
				Now, let \( h \in H \cap A^{x} \setminus H^{x} \). Then \( xhx^{-1} \notin H \), implying that \( h^{-1}x^{-1}H \neq x^{-1}H \). Since \( h \in A^{x} \), we have \( h^{-1}x^{-1}A = x^{-1}A \), which means that \( x^{-1}A \cap S = x^{-1}A \cap Hx^{-1}H \) contains more than one coset of \( H \), leading to a contradiction. Thus, we conclude that \( H \cap A^{x} \subseteq H^{x} \), as intended. 		
		Finally, we will show that condition (3) holds. If \( a \in A \cap A^{x} \), there is nothing to prove. So, we assume \( a \in A \setminus A^{x} \). Then \( ax^{-1}A \neq x^{-1}A \). Therefore, there exists \( h \in H \) such that \( hx^{-1}H \subseteq ax^{-1}A \cap Hx^{-1}H = ax^{-1}A \cap S \). Hence, \( xa^{-1}hx^{-1} \in A \), which means \( a^{-1}h \in A^{x} \). Consequently, \( h^{-1}a \in A^{x} \cap A \). This completes the proof of one direction.
			
		Conversely, assume that conditions (1)-(3) hold. We only need to prove that \( S \cap tA \) contains exactly one coset of \( H \) for each \( t \in G \setminus A \). Let \( t \notin A \). Then \( tA \subseteq AxA=Ax^{-1}A \), which means \( t = ax^{-1} \) for some \( a \in A \). There exists \( h \in H \) such that \( ha \in A \cap A^{x} \), implying that \( xhax^{-1} \in A \) or, equivalently, \( ax^{-1}A = h^{-1}x^{-1}A \). Thus, \( h^{-1}x^{-1}H \subseteq ax^{-1}A \cap S = tA \cap S \). 
				Now, we claim that \( tA \cap S = ax^{-1}A \cap S = h^{-1}x^{-1}H \). Suppose, for the sake of contradiction, that there exists \( h_0 \) such that \( h_0x^{-1}H \neq h^{-1}x^{-1}H \) and \( h_0x^{-1}H \subset ax^{-1}A \cap S \). Then \( hh_0 \not\in H \cap H^{x} = H \cap A^{x} \). This implies that \( h_0xA \neq h^{-1}x^{-1}A = ax^{-1}A \), a contradiction. This completes the proof.		
			\end{proof}

%\textbf{Acknowledgments.}

%%%%%%%%%%%%%%%%%%%%%%%%%%%%%%%%%%%%%%%%%%%%%%%%%%%%%%%%%%%%%%%%%%%%%%%%%%%%

\end{document}